\documentclass{birkjour_t2}
\usepackage{amssymb,amsmath,mathrsfs,amsfonts,amsthm,cases}

\usepackage{enumerate}
\usepackage{xcolor}
\usepackage[numbers]{natbib}
\usepackage{geometry}
\usepackage{setspace}
\usepackage{appendix}
\usepackage[colorlinks,
            linkcolor=teal,
            anchorcolor=blue,
            citecolor=purple]{hyperref}

\renewcommand{\labelenumi}{(\roman{enumi})}

\begin{document}

\newtheorem{theorem}{Theorem}[section]

\newtheorem{lemma}[theorem]{Lemma}
\newtheorem{prop}[theorem]{Proposition}
\newtheorem{cor}[theorem]{corollary}
\newtheorem{remark}{Remark}[section]
\newtheorem*{con}{Conjucture}
\newtheorem{assumption}{Assumption}
\newtheorem*{questionA}{Question}
\newtheorem*{thmA}{Theorem A}
\newtheorem*{thmB}{Theorem B}
\newtheorem{definition}{Definition}[section]

\def\p{\partial}

\def \pr1 {\partial_{R_1}}

\def \gr {\nabla_R}
\def \div {\text{div}\,}
\def \divr {\text{div}_R}
\def \d {\,\mathrm{d}}
\def \dx {\,\mathrm{d}x}
\def \dr {\,\mathrm{d}R}
\def \dri {\,\frac{\mathrm{d}R}{\psi_\infty}}
\def \dt {\frac{\d }{\d t}}
\def \la {\Lambda}
\def \ls {{\g^s}}
\def \lsm {{\g^{s-1}}}
\def \lsp {{\g^{s+1}}}

\def \xier {|\xi|^2}

\def \bl {\big\{}
\def \brr {\big\}}
\def \bbl {\Big\{}
\def \bbr {\Big\}}

\def \hshot {{s,\hd}^2}
\def \hthot {{H^2(\hd)}^2}
\def \hthtt {{H^2(\hdt)}^2}
\def \hsltt {{s,\h L^2}^2}
\def \htltt {{H^2(\h L^2)}^2}
\def \holtt {{H^1(\h L^2)}^2}
\def \ltltt {{\h L^2}^2}
\def \lthot {{\hd}^2}
\def \lthtt {{L^2(\hdt)}^2}
\def \hohot {{H^1(\hd)}^2}

\def \hyi {{{\dot{\mathcal H}}^1}}
\def \htnhot {{H^2(\h H^1)}^2}
\def \htnho {{H^2(\h H^1)}}

\def \hdyi {\dot{H}^1}
\def \hser {{s}^2}
\def \herer {{H^2}^2}
\def \hyier {{\hd}^2}
\def \lerer {{\h{L}^2}^2}
\def \hsoer {{H^{s-1}}^2}
\def \g {\nabla}

\def \hsho {{s,\hyi}}
\def \htho {{H^2(\hd)}}
\def \htht {{H^2(\hdt)}}
\def \hslt {{s,\ler}}
\def \hsplt {{H^{s+1}(\h L^2)}}
\def \htlt {{H^2(\h L^2)}}
\def \hoho {{H^1(\hd)}}
\def \ltho {\hyi}
\def \ltht {{L^2(\hdt)}}
\def \holt {{H^1(\h L^2)}}
\def \ltlt {{\h L^2}}
\def \lilt {{L^{\infty}(\h L^2)}}
\def \liho {{L^{\infty}(\hd)}}
\def \liht {{L^{\infty}(\hdt)}}

\def \ep {\varepsilon}

\def \hs {s}
\def \hso {s-1}
\def \her {{H^2}}
\def \ler {{{\h L}^2}}
\def \linf {{L^\infty}}

\def \le {\lesssim}

\def \Ltilt { L^\infty( [0,T]; \hlt ) }
\def \Lttlt { L^2( [0,T]; \hlt ) }
\def \Lttho { L^2( [0,T]; \hd ) }
\def \Lttht { L^2( [0,T]; \hdt ) }
\def \Ltiho { L^\infty( [0,T]; \hd ) }

\def \ltilt { {L^\infty([0,T];  \hlt )} }
\def \lttlt {{ L^2([0,T];  \hlt ) }}
\def \lttho { {L^2([0,T];  \hd )} }
\def \Lttht { {L^2([0,T]; \hdt )} }
\def \ltiho {{ L^\infty([0,T];  \hd }) }

\def \r {\mathbb{R}}
\def \T {\mathbb{T}}
\def \D {\mathbb{D}}
\def \S {\mathbb{S}}
\def \P {\mathbb{P}}
\def \h {\mathcal}
\def \hlt {{\h L^2}}
\def \hd {{{\dot{\mathcal H}}^1}}
\def \hdt {{\dot{\mathcal H}}^2}
\def \pin {\psi_\infty}
\def \vp {\varPsi}
\def \buk {U^\kappa}
\def \uk {u^\kappa}
\def\bpk{\varPsi^\kappa}
\def\pk{\psi^\kappa}

\def \holder {H$\ddot{\text{o}}$lder's }
\def \poin {Poincar$\Acute{\text{e}}$ }
\def \m {{\mathfrak m}}

\def \tbuk {\lt {U^\kappa}}
\def \tuk {\lt {u^\kappa}}
\def \abuk {\lb {U^\kappa}}
\def \auk {\lb {u^\kappa}}

\def \ser {\sum\limits_{k=1}^2 }
\def \serik {\sum\limits_{i,k=1}^2 }
\def \seri {\sum\limits_{i=1}^2 }
\def \serj {\sum\limits_{j=1}^N }
\def \serl {\sum\limits_{l=1}^2 }
\def \serk {\sum\limits_{k=1}^2 }
\def \serij {\sum\limits_{i,j=1}^N }
\def \serjk {\sum\limits_{j,k=1}^N }
\def \serjkl {\sum\limits_{j,k,l=1}^N }
\def \ssan {\sum\limits_{k=1}^3 }
\def \sij {\sum\limits_{i,j=1}^N }
\def \sj {\sum\limits_{j=1}^N }
\def \serjkn {\sum\limits_{j,k=1}^N }

\def \ix {\int_{\mathbb{R}^N}}
\def \is { \int_{S(t)} }
\def \it { \int_0^t }
\def \intox {\int_{\mathbb{R}}}
\def \intoxr {\iint\limits_{{\,\mathbb{R}} {B}}}
\def \intnx {\int_{\mathbb{R}^N}}
\def \intr {\int_B}
\def \ixr {\iint_{\mathbb{R}^{N}\!\times\!{B}}}

\def \ey{\cal E_1 }
\def \ee{\cal E_2}

\def \intdr {\iint\limits_{{\,\Omega} {B}}}
\def \intd {\int_{\Omega}}

\def \hu {\hat u}
\def \hbu {\bar {\hat u}}
\def \F {\cal F}
\def \hp {\hat \psi}
\def \hbp {\bar{\hat \psi}}
\def \re {\cal R e}

\def \rjk {R_jR_k}

\def \cu {\mathbf}
\def \cal{\mathcal}

\def \t {\tilde}
\def \lt {\widetilde}
\def \b {\bar}
\def \lb {\overline}

\bibliographystyle{abbrvnat}

\title[Euler-FENE dumbbell]{Stability and large-time behavior for the N-Dimensional Euler-FENE dumbbell model near an equilibrium}


\author[Z. Yao]{Zheng-an Yao}
\address{
Institute of Advanced Studies Hong Kong,
Sun Yat-sen University,
Hongkong  999077, China}
\email {mcsyao@mail.sysu.edu.cn}

\author{Ruijia Yu}
\address{ 
School of Mathematics,
Sun Yat-sen University,
Guangzhou  510275,
  China}
\email {yurj5@mail2.sysu.edu.cn}

\begin{abstract}
This paper studies the $N$-dimensional FENE dumbbell model without velocity dissipation, focusing on the stability and decay of perturbations near the steady solution $(0,\pin)$.
Due to the lack of velocity dissipation, the above problems are highly challenging. 
In fact, without coupling, the corresponding $N$-dimensional Euler equation near $u=0$ is well known to be unstable.
To overcome this difficulty, we analyze the   wave structure arising in the system governing perturbations around the steady state, which originates from the equilibrium configuration and the coupling effects.
This   wave structure enables us to establish the global stability in the $H^s$-type Sobolev norms.
Also, we highlight the critical role of   wave structure in the decay estimates of the Euler-FENE dumbbell model. By combining this property with the Fourier splitting method, we derive the decay rate which is identical to that of the general FENE dumbbell with velocity dissipation.
\end{abstract}


\maketitle

  \section{Introduction}
Fluid-polymer coupled systems are fundamental in physics, chemistry, and biology, drawing significant research interest \cite{bird1,masmoudi-cpam}.  One of these models is the finite extensible nonlinear elastic (FENE) dumbbell model.
 In this model, a polymer is depicted as an "elastic dumbbell," consisting of two "beads" linked by a spring, which is represented by a vector 
$R$ (see \citet{bird1,bird2}, \citet{doi}, \citet{ottinger}).
 At the fluid level, the FENE dumbbell model combines the Navier-Stokes equations describing fluid velocity with a Fokker-Planck equation governing the dynamics of polymer distribution in the liquid medium.
 The micro-macro FENE dumbbell model reads as follows
  \begin{equation}\label{original}
  \begin{cases}
          \p_t u+(u\cdot \nabla) u= \mu \Delta u-\nabla p+\div \tau, \quad \div u = 0,\\
          \p_t \Psi + u\cdot \nabla \Psi
          = \nu \Delta \Psi 
          +\divr\left[
          -\nabla u\cdot R\Psi+\beta\gr \Psi +\gr\, \mathcal U \Psi
          \right ],\\
          \tau_{l,m}=\intr (R_l \p_{R_m} \mathcal U)\Psi(x,R,t)\dr,\\
          u(x,0)=u_0,\quad \Psi(x,R,0)=\Psi_0,\\
          (\beta \gr \Psi +\gr\, \mathcal U \Psi)\cdot n = 0 \quad \text{on} \ \p B(0,R_0).
  \end{cases}
  \end{equation}
  In (\ref{original}), $u(x,t)$ denotes the velocity of the polymeric liquid, and $\Psi(x,R,t)$ denotes the distribution function of the internal configuration, where $x\in \r^N$, $N\geq 2$. The polymer elongation $R$ is bounded in a ball $B(0,R_0)$, which indicates that the extensibility of the polymers is finite.
  Also, the potential $\h U$ is given by $\h U(R)=-k \log ( 1-\frac{|R|^2}{R_0^2} )$ for some $k>0$. 
  In addition, $\nu$ is the center-of-mass diffusion rate of the polymer, and $\beta$ relates to the Boltzmann constant and temperature. 

  As in \cite{masmoudi-cpam}, to ensure the conservation of $\Psi$, we add an additional boundary condition, namely
  $$
  (-\nabla u\cdot R\Psi+\beta\gr \Psi +\gr\, \mathcal U \Psi)\cdot n=0 \quad \text{ on }\p B(0,R_0).
  $$
  This boundary condition implies that $\Psi=0$ on $\p B(0, R_0)$, and if $\intr \Psi_0( x, R ) \dr=1$, then for all $(x,t)$, we have $\intr \Psi( x,R,t ) \dr=1$.

 In this paper, we will set $\beta=1$ and $R_0=1$. 
  Notice that $u=0$ and 
  $$
  \pin(R)=\frac{e^{-\h U(R)}}{\intr e^{-\h U(R)}
\dr }=\frac{(1-|R|^2)^k}{\intr (1-|R|^2)^k \dr}
  $$
  is a trivial solution of (\ref{original}), and we denote 
  $$
  \h L\Psi=\divr ( \gr \Psi +\gr\, \mathcal U \Psi )
  =\divr \left( \pin \gr \frac{\Psi}{\pin} \right).
  $$

 There are plenty of mathematical results about the FENE dumbbell model. 
 If $\nu=0$ and $\mu>0$, \citet{renardy} demonstrated the local well-posedness in Sobolev space for potentials of the form $\h U(R)=(1-|R|^2)^{1-\sigma}$ with some $\sigma>1$. 
 Later, \citet{li-cmp} proved the local existence when $R$ is taken in the whole space and under certain growth conditions on the potential.
 Also, \citet{jourdain} analyzed the local
existence of a stochastic differential equation with potential $\h U(R)=-k \log(1-|R|^2)$ and established local existence for Couette flow, under the assumption $b=2k>6$.
\citet{zhang-arma} proved the local well-posedness of the three-dimensional FENE model in weighted Sobolev spaces for $b=2k > 76$. 
\citet{masmoudi-cpam} developed some novel Hardy-type inequalities to handle the singular term $\div \tau$, and proved local well-posedness and global well-posedness under small assumption near the equilibrium for the FENE model when $b=2k > 0$.
Also near the equilibrium, \citet{lin-cpam} proved global existence under certain constraints on the potential.
Regarding the global existence of weak solutions in $L^2$, \citet{masmoudi-invension} derived results based on some entropy conditions.
Later, Luo and Yin \cite{luo-arma,luo-advance} overcame the Liouville problem and established the $L^2$ decay for the FENE model.

On the other hand, the center-of-mass diffusion of polymers is a well-established physical phenomenon and has been extensively studied.
 In this case, namely $\mu>0$ and $\nu>0$, the global existence of weak solution of the FENE model incorporating center-of-mass diffusion was demonstrated in \cite{barrett}.
Also, \citet{barrett-cutoff} introduced a  “microscopic” cut-off function in the drag term of the Fokker-Planck equation and proved the existence of global-in-time weak solutions to a mollification model with a general class of spring-force potentials, including the FENE potential.
Subsequently, \citet{barrett-mmm} removed the cut-off in \cite{barrett-cutoff} and extended the results to the case of bead-spring chain models.

Regarding the decay of the FENE dumbbell model near the steady state, numerous important results have been obtained when the velocity equation exhibits full dissipation.
 In 2006, \citet{schonbek} investigated the $L^2$ decay of the velocity for the co-rotation FENE dumbbell model, and obtained the decay rate $(1+t)^{ -\frac{N}{4}+\frac{1}{2} }$.
 Luo and Yin \cite{luo-arma,luo-advance} enhanced the $L^2$ decay results developed in \cite{schonbek}  by Fourier methods, and obtained the optimal decay rates for the co-rotation case.
 Later, \citet{luo2024optimal} refined the decay result in \cite{luo-arma} for the general case and obtained the optimal decay rates.
 However, in the absence of explicit velocity dissipation, establishing the decay properties of the FENE dumbbell model becomes highly challenging.

  In this paper, we explore the   wave structure of the Euler-FENE dumbbell model generated by the steady-state $(0,\pin)$ and the coupling and interaction. For this purpose, we set $\mu=0$ and $\nu>0$. Also, we denote $\psi=\Psi-\pin$.
 Since $\nabla_x \pin=0$, we have 
 $$
 \div \tau=\div \intr (R\otimes\gr \h U)\Psi \dr = \div \intr (R\otimes\gr \h U)\psi \dr.
 $$
 Hence, we may assume that
 $$
 \tau= \intr (R\otimes\gr \h U)\psi \dr.
 $$
 Now, we write $\psi_0 :=\Psi_0-\pin$, and the equation governing the perturbation $(u,\psi)$ of $N$-dimensional Euler-FENE dumbbell model reads as follows
  \begin{equation}\label{fene}
      \begin{cases}
          \p_t u+(u\cdot \nabla) u= -\nabla p+\div \tau, \quad \div u = 0,\\
          \p_t \psi + u\cdot \nabla \psi
          = \nu\Delta\psi 
          +\divr\left[
          -\nabla u\cdot R(\psi+\pin)\right ]+\h L \psi
          ,\\
          u(x,0)=u_0,\quad \psi(x,R,0)=\psi_0,\\
          \pin\gr\frac{\psi}{\pin}\cdot n = 0 \quad \text{on} \ \p B(0,1).
      \end{cases}
  \end{equation}
 
This paper comprehensively investigates (\ref{fene}), focusing on the small data global existence and stability, and the decay estimates. The first main result reads as follows.
 
  \begin{theorem}\label{thm-exist}
      Suppose that $k>1$, $s >\frac{N}{2}+1$, there exists a small constant $\varepsilon>0$ such that for $u_0 \in H^s(\r^N)$, $\div u_0 = 0$ and $\psi_0 \in H^s(\r^N; \h L^2)$, if
      \begin{equation}
          ||u_0||_s+||\psi_0||_{s, \h L^2}\leq \varepsilon,
      \end{equation}
      then (\ref{fene}) has a unique global solution $(u,\psi)$. 
      In addition, for all $t > 0$, the following estimates hold:
      \begin{equation*}
      \begin{aligned}
      &||u(t)||_s^2 + ||\psi(t)||_\hsltt
      + \int_0^t
      \bbl
      ||\g u(s)||_{s-1}^2
      +\nu ||\g\psi(s)||_\hsltt
      +||\psi(s)||_\hshot
    \bbr
    \d s\leq C \varepsilon^2.
      \end{aligned}
    \end{equation*}  
  \end{theorem}
  
  \begin{remark}
    Without the coupling of the polymer flow, (\ref{fene}) reduces to the incompressible Euler equation, which is not stable near $u=0$ even in 2-dimensions \cite{kiselev}. Hence, the   wave structure induced by the equilibrium and by the coupling and interaction plays a crucial role in ensuring the stability.
    To the best of our knowledge, this work provides the \textbf{first rigorous stability result} for the Euler–FENE dumbbell model.
\end{remark}

Next, we demonstrate the importance of the   wave structure in the study of decay estimates of (\ref{fene}).
  \begin{theorem}\label{thm-decay}
      Suppose that the assumptions of Theorem \ref{thm-exist} hold, with $\ep$ chosen to be smaller if necessary.
    Let the initial data $(u_0,\psi_0)$ also satisfy
    $$
    ||u_0||_{L^1}+ ||\psi_0||_{L^1(\ler)} \leq \varepsilon. 
    $$
    Then the following decay estimates hold:
    \begin{equation}
    \begin{aligned}
        &||u(t)||_1\leq \varepsilon C (1+t)^{-\frac{N}{4}},\\
         &||\psi(t)||_{1,\ler} \leq \varepsilon C (1+t)^{-\frac{N+2}{4}}.
    \end{aligned}
    \end{equation}
  \end{theorem}

\begin{remark}
    The obtained velocity decay rate aligns with that of the viscoelastic FENE dumbbell models, namely (\ref{original}) with $\mu>0$ and $\nu=0$ \cite{luo2024optimal}.
    Also, due to the   wave structure, the velocity governed by this forced incompressible Euler equation exhibits the same decay rates as the incompressible Navier-Stokes equations.
    We suggest that this decay rate is optimal, as it coincides with that of the heat equation.
\end{remark}
 
\subsection{Main difficulties and strategies}
In this part, we explain the main difficulties we encountered during our proofs and our main strategies. 
\subsubsection{Global existence and stability}
In this part, a major obstacle arises from the absence of the velocity dissipation. This greatly complicates the analysis and makes the issues of global existence and stability non-trivial.

 Without coupling the equation of polymer density, (\ref{fene}) becomes the Euler equation:
  
    \begin{equation}\label{ns-2d}
      \begin{cases}
          \p_t u+(u\cdot \nabla) u= -\nabla p,\\
          \div u = 0,
      \end{cases}
  \end{equation}
 \citet{kiselev} proved that the gradient of vorticity could grow double exponentially in time. Consequently, in general, in the absence of coupling with the polymer equation, (\ref{ns-2d}) is not stable near $u=0$.
Therefore, when analyzing the stability of the Euler–FENE dumbbell model, the absence of velocity dissipation prevents the following standard $H^s$ energy framework from closing the estimates:
   \begin{equation*}\label{enu1}
     \begin{aligned}
         \ey(t)=&\sup\limits_{0\leq s\leq t}||u(s)||_s^2
         +\sup\limits_{0\leq s\leq t}||\psi(s)||_{s,\ler}^2
         +\int_0^t 
         \bbl
        \nu||\g\psi(s)||_{s,\ler}^2+ ||\psi(s)||_{s,\hyi}^2 
         \bbr
         \d s,
     \end{aligned}
 \end{equation*}
  but also need to explore the   wave structure of (\ref{fene}) and discover the additional regularity of $u$.
   To better illustrate the wave structure of (\ref{fene}), we consider the two-dimensional case.
  Let $\P = I - \nabla\Delta^{-1}\div$ denote the  Leray projection onto divergence free vector fields. By applying the Leray projection $\P$ to the first equation of (\ref{fene}) to eliminate the pressure term, we obtain
  \begin{equation}\label{linear-u}
      \p_t u  = \P\div \tau +\cal N_1,\quad
      \cal N_1=- \P ( u\cdot\nabla u ).
  \end{equation}
  Now, by integrating the second equation of (\ref{fene}) with $R\otimes\gr \h U$, and applying $\P \div$, we have
  \begin{equation}\label{divtao}
  \begin{aligned}
      \p_t\P \div\tau =&\nu \Delta \P \div \tau 
      +\P \div \intr \h L\psi  R\otimes \gr \h U \dr \\
      &+\P \div \intr \divr(-\nabla u\cdot R\pin )R\otimes \gr \h U \dr 
      + \cal N_2,
  \end{aligned}
  \end{equation}
  where
    \begin{equation*}
    \begin{aligned}
      \cal N_2=& - \P \div\left\{u\cdot\nabla \intr \psi R\otimes \gr \h U \dr\right\} \\
      &+ \P \div \intr \divr(-\nabla u\cdot R\psi) R\otimes \gr \h U \dr.
    \end{aligned}
    \end{equation*}
 Using the axisymmetric of $\pin$ and $\div u =0$, we have
 \begin{equation*}\label{laplace-u}
 \begin{aligned}
     &\P \div \intr \divr(-\nabla u\cdot R\pin) R\otimes \gr \h U \dr\\
     =&
     \P\begin{pmatrix}
         c_1(k)\p_1^2u_1+2c_2(k)\p_1\p_2 u_2+c_2(k)\p_2^2 u_2\\
         c_2(k)\p_1^2 u_2+2c_2(k)\p_1\p_2 u_1 +c_1(k)\p_2^2 u_2 
     \end{pmatrix}\\
     =&
     \P\begin{pmatrix}
         c_1(k)\p_1^2 u_1-2c_2(k)\p_1^2u_1+c_2(k)\p_2^2 u_2\\
         c_2(k)\p_1^2 u_2-2c_2(k)\p_2^2 u_2+ c_1(k)\p_2^2 u_2
     \end{pmatrix}
     =c_2(k) \Delta u,
 \end{aligned}
 \end{equation*}
 where 
 $$
 c_1(k)=-2k\intr \frac{R_1^3\p_{R_1}\pin} {1-|R|^2} \dr,\quad
 c_2(k)=-2k\intr \frac{R_1^2R_2\p_{R_2}\pin} {1-|R|^2} \dr, $$
 and $ c_1(k) $, $c_2(k)$ are integrable if $k>1$, and satisfy
  $c_1(k)=3c_2(k)>0$.

 Hence, by differentiating (\ref{linear-u}) and (\ref{divtao}) in time, we obtain that 
 \begin{equation}\label{wave-1}
     \begin{cases}
         \displaystyle\p_t^2 u- \nu   \p_t \Delta u 
         -c_2(k)\Delta u 
         + \P \div \intr \h L\psi  R\otimes \gr \h U \dr =\cal N_3,\\
         \displaystyle\p_t^2 \Upsilon
         -\nu \p_t \Delta \Upsilon
         -c_2(k)\Delta \Upsilon
         -\p_t \P \div \intr \h L\psi  R\otimes \gr \h U \dr=\cal N_4,
     \end{cases}
 \end{equation}
  where
  $$
  \cal N_3= (\p_t-\nu\Delta)\cal N_1 +\cal N_2
  ,\quad \cal N_4=c_2(k)\Delta \cal N_1 + \p_t \cal N_2.
  $$
In particular, (\ref{wave-1})  reveals the dissipation mechanism for $u$, which plays a crucial role in ensuring the velocity stability. More precisely, thanks to the wave structure (\ref{wave-1}), we can extract the temporal integrability of $\g u$, which makes it promising to give the uniform bound of the following energy structure:
 \begin{equation*}\label{e2-intro}
     \ee(t)=\int_0^t ||\g u(s)||_{s-1}^2 \d s.
 \end{equation*}

 In addition to understanding the   wave structure of $u$ from the wave structure, there is another simple way to derive (\ref{e2-intro}) without introducing the singularity. 
 In fact, by integrating the second equation of (\ref{fene}) with $R_kR_l$, where $k,l=1,\cdots,N$, we obtain
 \begin{equation}\label{du-temp1}
     \begin{aligned}
             &\intr \p_t \psi R_kR_l \dr +\intr u\cdot\nabla \psi R_kR_l \dr\\
            =&\nu \intr \Delta \psi R_kR_l \dr
             + \intr \divr (\pin \gr \frac{\psi}{\pin})R_kR_l \dr\\
             &+ \intr \divr (-\nabla u \cdot R \psi)R_kR_l \dr
             +\intr \divr (-\nabla u \cdot R \pin)R_kR_l \dr.
     \end{aligned}
 \end{equation}
 Due to the symmetry of $\pin$, we deduce that the last term of (\ref{du-temp1}) is equivalent to a constant $\cal C$ times $\D u$:
 \begin{equation*}\label{du-temp2}
     \begin{aligned}
        &\intr \divr (-\nabla u \cdot R \pin)R_kR_l \dr\\
         =  &\serij \intr \p_ju_iR_j\pin \p_{R_i}(R_kR_l) \dr\\
         =& \serj \intr
         \left\{
         \p_j u_k R_j \pin R_l+ \p_j u_l R_j\pin R_k 
         \right\}
         \dr\\
         = &\Big (2\intr R_1^2\pin \dr\Big) [\D u]_{k,l},
     \end{aligned}
 \end{equation*}
 where $[\D u]_{k,l}=\frac{1}{2}( \p_k u_l+ \p_l u_k )$.
  By plugging the above equation into (\ref{du-temp1}), we find
  \begin{equation}\label{du}
  \begin{aligned}
      \cal C [\D u]_{j,k}=&
       \intr \p_t \psi R_kR_l \dr +\intr u\cdot\nabla \psi R_kR_l \dr\\
       &-\nu \intr \Delta \psi R_jR_k \dr
        -\intr \divr (\pin \gr \frac{\psi}{\pin})R_jR_k \dr\\
       & - \intr \divr (-\nabla u \cdot R \psi)R_jR_k \dr,
  \end{aligned}
  \end{equation}
  where $\cal C$ is a constant.
  Since $\div u=0$, we have $||\D u||_s=||\g u||_s$. 
  Hence, the time integrability of $\nabla u$ can be estimated by the terms on the right-hand side of (\ref{du}). The reasoning above explains our strategy on how to prevent the growth of the Sobolev norms of velocity by exploiting the stabilizing effect of $\psi$ on the fluid.

\subsubsection{Decay}
The absence of velocity dissipation also causes significant difficulty when considering the decay property, especially for the velocity. 
To overcome this difficulty, we take advantage of the   wave structure of the velocity, and combine (\ref{du}) with the $L^2$-estimates.
However, the involvement of (\ref{du}) induces linear terms with time-derivative, namely
$$
\p_t\ixr ( \psi R\otimes R):\D u\dr\dx. 
$$
To handle the above linear term, we innovatively introduce the $\dot H^1$-terms 
 and construct a two-layer estimate (see (\ref{equ-decay})). By combining (\ref{equ-decay}) with Fourier splitting estimates, we obtain the decay rate that is identical to the heat equation. 

\medskip
The rest of this paper is organized as follows. 
 In Section  \ref{sec-pre}, we introduce some notations and list several Lemmas that will be frequently used.
 Section \ref{sec-global existence} proves the global existence and stability.
 Section \ref{sec-decay} is devoted to the decay estimates of (\ref{fene}).
  
  \section{Preliminaries}\label{sec-pre}
   In this section, we introduce some notations and lemmas that we shall use throughout the paper.
   
    \subsection{Notations}
    In this paper, we will use the following notations. We use $f\le g$ to denote $f\leq C g$. 
        For the norm of Sobolev spaces in $x$-variable, we denote
$$
||f||_m :=||f||_{H^m},\quad
||f|| :=||f||_0,\quad
|||f|||:=||f||_{L^\infty}.
$$
    Also, we use the abbreviation $B=B(0,1)$, and $\p_{R_i}$, $\divr$ and $\gr$ denote the derivative, divergence and gradient in $R$-variable, respectively.
    We define the following Hilbert spaces in $R$-variable
    $$
    \begin{aligned}
    \h L^2 &= L^2({\dr/\pin })=\left\{
    \psi \ \Big| \ |\psi|^2_{\h L^2}=\intr |\psi|^2\dri<\infty
    \right\},\\
    \hd &= \left\{
    \psi \ \Big| \ 
   |\psi|_\hd^2 = \intr \pin \Big | \gr\frac{\psi}{\pin} \Big |^2 \dr < \infty
\right\},
    \end{aligned}
    $$

Also, for $\alpha\in \mathbb N^N$, $\g^\alpha$  denotes  $\alpha_1$ derivatives in $x_1$, $\cdots$, $\alpha_N$ derivatives in $x_N$.
With the above notations, we are able to 
 define the norms involving $x$ and $R$. For $s\geq 0$,
$$
\begin{aligned}
||\psi||_\hsltt&=\sum\limits_{|\alpha|\leq s}\ixr \left|\g^\alpha \psi\right|^2\dri\dx,\\
||\psi||_\hshot&=\sum\limits_{|\alpha|\leq s} \ixr \pin \left |\g^\alpha \gr \frac{\psi}{\pin}\right|^2\dr\dx.
\end{aligned}
$$
Also, we set
$$
||f||_\ler:=||f||_{0,\ler},\quad
||f||_\hyi:=||f||_{0,\hyi},
$$
and we use abbreviation $\iint:= \ixr$.

   We define $\Lambda=(-\Delta)^\frac{1}{2}$ to denote the Zygmund operator, where the fractional Laplacian $(-\Delta)^{s}$ is given by the Fourier transform
   $$
   \widehat{(-\Delta)^s}f(\xi)=|\xi|^{2s}\hat f(\xi).
   $$

\subsection{inequalities in $R$-variable}
The first Lemma is the \poin inequality in $R$-variable.
  \begin{lemma}[\cite{masmoudi-cpam}]\label{le-poincare-r}
      Assume that $\psi \in \hyi$ and $\intr \psi=0$, then
      \begin{equation*}
          \intr \frac{\psi^2}{\pin} \dr
          \le \int \pin \left| \gr  \frac{\psi}{\pin}\right|^2 \dr.
      \end{equation*}      
  \end{lemma}

 To deal with the singular term $\div \tau$, the main tool is the following Hardy-type inequality. 
 
  \begin{lemma}[\cite{masmoudi-cpam}]\label{le-tau}
      Assume that $\psi \in \ler\cap \hyi$ and $\intr \psi=0$, then
\begin{equation*}
          \left(\intr \frac{|\psi|}{1-|R|}\dr \right)^2
      \le \intr \pin 
      \left |
      \gr\frac{\psi}{\pin}
      \right|^2 \dr.
\end{equation*}
Morevoer, if $k>1$, then  
            \begin{equation*} \label{ineq-hardy}
          \intr \frac{\psi^2}{\pin (1-|R|)^2} \dr
          \le\intr \pin 
      \left |
      \gr\frac{\psi}{\pin}
      \right|^2 \dr.
      \end{equation*}
  \end{lemma}

   \section{Global Existence}\label{sec-global existence}

  In this section, we Investigate the small data global existence and stability of the N-dimensional Euler-FENE dumbbell. Thanks to the   wave structure of (\ref{fene}), we can discover the additional time integrability of $u$. Recall that we define  the energy functionals as follows
  \begin{equation*}\label{nde1}
      \ey (t)=\sup\limits_{0\leq s\leq t}||u(s)||_{s}^2
      +\sup\limits_{0\leq s\leq t}||\psi(s)||_\hsltt
     + \int_0^t \bbl\nu ||\nabla\psi(s)||_\hsltt
     + ||\psi(s)||_\hshot \bbr \d s,
 \end{equation*}
 \begin{equation*}\label{nde2}
      \ee(t)=\int_0^t ||\nabla u(s)||_{s-1}^2 \d s.
 \end{equation*}

 \begin{proof}[Proof of Theorem \ref{thm-exist}]
    By the bootstrap argument (see e.g. P.21 of \cite{tao}), it suffices to prove the following two inequalities:
     \begin{align}
         &\ey (t) \leq  \ey (0) + C  \ey^{\frac{3}{2}} (t) + C  \ee^{\frac{3}{2}}(t),\label{nde1-inequ}\\
         &           \ee(t) \leq C  \ey (0) +  C\ey (t) + C  \ey ^{\frac{3}{2}}(t) + C \ee^{\frac{3}{2}}(t).\label{nde2-inequ}
     \end{align}
     
     We begin with (\ref{nde1-inequ}). By standard $L^2$ energy method, we deduce from integrating by parts that
    $$
    \begin{aligned}
          &\frac{1}{2}\dt (||u||^2 +||\psi||_\ler ^2)
          + \nu ||\nabla \psi||_\lerer+||\psi||_\hyier\\
          = &\iint \nabla u\cdot R \psi\cdot \gr \frac{\psi}{\pin}  \dr\d x\nonumber\\
          \le & ||\nabla u||\,|||
          \psi|||_\ler ||\psi||_\hyi \nonumber
          \leq  ||\psi||_{s,\ler}
          \bl ||\nabla u||^2+||\psi||_\hyier\brr.
    \end{aligned}
    $$
         Here we use the fact that, by integrating by parts and $-\dfrac{\p_{R_i}\pin}{\pin}=\p_{R_i} \h U$, 
     \begin{equation*}\label{linearpart-equal}
     \begin{aligned}
         &\iint \divr (-\g u\cdot R \pin)\psi \dri\dx\\
         =& -\iint \div u\pin\psi  \dr \dx 
         +\sum\limits_{i,j=1}^N \iint \p_j u_i R_j\p_{R_i}\h U  \psi \dr \dx\\
         =  & - \ix \div\tau\cdot u \dx.
     \end{aligned}
    \end{equation*}
    
     Now we consider the $\dot H^s$ estimate. Direct computation shows
      \begin{equation}\label{nde1-estimate}
      \begin{aligned}
               & \frac{1}{2}  \dt (||\ls u||^2 
                +||\ls \psi||_\lerer)
          +  \nu ||\ls \nabla  \psi||_\lerer
          +  ||\ls\psi||_\hyier = I_1 + I_2 + I_3.
      \end{aligned}
      \end{equation}
      where
      $$
      \begin{aligned}
          I_1:=& - \intnx  \ls(u\cdot\nabla u)\cdot\ls u \d x ,\\
          I_2:=& - \iint \ls(u\cdot\nabla\psi)\ls \psi \dri \d x,\\
          I_3:= & \iint \ls \divr ( -\nabla u \cdot R \psi)\ls \psi \dri \d x.
      \end{aligned}
      $$
 Using $\div u=0$, we can bound $I_1$ and $I_2$ as follows
 \begin{equation*}
      \begin{aligned}
     |I_1|&=
     \left |
     \,\intnx  (\ls(u\cdot\nabla u)-u\cdot\nabla \ls u)\ls u\d x
     \right|
      \le ||\nabla u||_{s-1}^2 ||\ls u||,\\
     |I_2|&=
     \left |
    \, \iint  (\ls(u\cdot\nabla \psi)-u\cdot\nabla \ls \psi)\ls \psi \dri\d x
     \right|
       \le ||u||_s ||\psi||_\hshot.
     \end{aligned}
 \end{equation*}
 Finally, we deduce from integrating by parts, $\div u=0$ and Lemma \ref{le-tau} that
 \begin{equation*}
 \begin{aligned}
     |I_3|
     \le & 
     \left |
     \, \iint  \lsm\left[ \nabla u\cdot R
     \left(\pin^{\frac{1}{2}}\gr
     \frac{\psi}{\pin}\right)
     \right]\pin^{-\frac{1}{2}}\lsp \psi
     \dr \d x
     \right| \\
     & + 
     \left |
     \, \iint  \lsm\left[ \nabla u\cdot R
     \left(\pin^{-\frac{1}{2}}
     \frac{\psi}{1-|R|}\right)
     \right]\pin^{-\frac{1}{2}}\lsp \psi
     \dr \d x
     \right|  \\
     \leq & 
     ||\g u||_{s-1} || \psi ||_{s,\hyi}
     ||\lsp \psi||_\ltlt\\
     \leq &||u||_\hs
     \bl
     ||\psi||_\hshot + ||\nabla\psi||_{s,\ler}^2
     \brr.
 \end{aligned}
 \end{equation*}
     Plugging the upper bounds of $I_1$, $I_2$ and $I_3$ into (\ref{nde1-estimate}), integrating in time and invoking  the norm equivalence, we find
\begin{align*}
    &||u||_\hser+||\psi||_\hsltt
    +\nu \int_0^t \bbl ||\nabla \psi(s)||_\hsltt +||\psi(s)||_\hshot \bbr \d s\\
    \leq  & ||u_0||_\hser+||\psi_0||_\hsltt
    + C \sup\limits_{0\leq s\leq t} ||u||_\hs 
    \int_0^t ||\nabla \psi(s)||_\hsltt +||\psi(s)||_\hshot \d s\\
    &+ C \sup\limits_{0\leq s\leq t} ||u||_\hs 
    \int_0^t \bbl ||\nabla u(s)||_{s-1}^2 \bbr \d s\\
    \leq & \ey (0) + C  \ey (t)^{\frac{3}{2}} + C  \ee(t)^{\frac{3}{2}},
\end{align*}
and (\ref{nde1-inequ}) is proved.

 Now we turn to take advantage of wave structure and prove (\ref{nde2-inequ}). 
 We deduce from the equation of $\D u$ in (\ref{du}) that
      \begin{equation}\label{Du-define-nd}
      \begin{aligned}
         ||\D u||_{s-1}^2 
         = & (\D u,\D u)+ \serl(\lsm\D u,\lsm\D u)
        = \sum\limits_{i=1}^5 J_i,
        \end{aligned}
        \end{equation}
    where
    $$
    \begin{aligned}
        J_1:=&\serjkn \iint \dt \psi R_j R_k [\D u]_{j,k}\dr \d x\\
       & + \serjkn \iint \dt \lsm \psi R_j R_k \lsm[\D u]_{j,k}\dr \d x,\\
        J_2 :=& \serjkn \iint u\cdot \nabla \psi \rjk [\D u]_{j,k}\dr \d x\\
        & + \serjkn \iint \lsm( u\cdot \nabla \psi \rjk) \lsm [\D u]_{j,k}\dr \d x,  \\
        J_3:= &\serjkn \iint \Delta \psi \rjk [\D u]_{j,k}\dr \d x\\
        & + \serjkn \iint \Delta \lsm \psi \rjk \lsm [\D u]_{j,k}\dr \d x,  \\
        J_4:=& \serjkn \iint \divr(-\nabla u \cdot R \psi)\rjk [\D u]_{j,k}\dr \d x\\
        &   +\serjkn\iint \lsm\divr (-\nabla u \cdot R \psi)\rjk \lsm [\D u]_{j,k}\dr \d x, \\
        \end{aligned}
        $$
        $$
        \begin{aligned}
        J_5:=& \serjkn  \iint \divr \Big( \pin \gr \frac{\psi}{\pin}\Big )\rjk[\D u]_{j,k}\dr \d x\\
        & +\serjkn \iint \lsm\divr \Big( \pin \gr \frac{\psi}{\pin}\Big )\rjk\lsm[\D u]_{j,k}\dr \d x.
    \end{aligned}
    $$

      To estimate $J_1$, we split it into four terms $J_1=J_{1,1}+J_{1,2}+J_{1,3}+J_{1,4}$, where
    \begin{align*}
    J_{1,1}&:=\serjkn \dt \iint \psi \rjk [\D \P u]_{j,k} \dr \d x ,\\
    J_{1,2}&:=- \serjkn \iint \psi\rjk \dt [\D \P u]_{j,k} \dr \d x,\\
    J_{1,3}&:=\serjkn \dt \iint  \lsm \psi \rjk \lsm [\D \P u]_{j,k} \dr \d x, \\
    J_{1,4}&:=- \serjkn \iint \lsm \psi\rjk \dt \lsm [\D\P u]_{j,k} \dr \d x.
     \end{align*}

       By integrating in time, we have
     \begin{align*}
         \left |\int_0^t J_{1,1} \d s \right | 
         &\leq  \serjkn \left| \iint \psi \rjk [\D \P u]_{j,k} \dr \d x  \right |
         +||\psi_0||_\ltlt||u_0||_\hyi\\
         &\leq ||u||_1^2 + ||\psi||_\ltltt  
         + ||u_0||_1^2 +  ||\psi_0||_\ltltt ,\\
        \left |\int_0^t J_{1,3} \d s \right | 
         &\leq  \serjkn \left| \iint  \lsm \psi \rjk \lsm  [\D \P u]_{j,k} \dr \d x  \right |
         +||\lsm\psi_0||_\ltlt||\ls u_0|| \\
         &\leq ||\ls u||^2 + ||\lsm \psi||_\ltltt
         +||\ls u_0||^2 + ||\lsm \psi_0||_\ltltt,
     \end{align*}   
     By substituting the first equation of (\ref{fene}) into $J_{1,2}$ and using Lemma \ref{le-poincare-r}, we obtain
     \begin{align*}
         J_{1,2}
         = & \serjkn \iint \psi \rjk 
         \left[
         \D (u\cdot \nabla u) -\D \P \div \tau 
         \right ]_{j,k} \dr\d x \\
         \leq & ||\psi||_\ltlt ||\nabla u||_{s-1}^2
         +||\psi||_{2,\hyi}^2.
     \end{align*}
     Using integration by part, $J_{1,4}$ can be treated similar to $J_{1,2} $:     
     \begin{align*}
         J_{1,4}
         = & \serjkn \iint \lsm  \psi \rjk 
          \lsm \left[
         \D (u\cdot \nabla u) -\D \P \div \tau 
         \right ]_{j,k} \dr\d x \\
         \leq & ||\psi||_{s,\ler} ||\nabla u||_{s-1}^2
         +||\lsp \psi||_\ltltt + ||\ls \psi||_\lthot.
     \end{align*}
     Summing up $N_{1,i}(i=1,2,3,4)$ and integrating in $t$ yields
     \begin{equation}
     \begin{aligned}\label{n1}
        \int_0^t J_1(s) \d s
         \le & ||u_0||_\hser + ||\psi_0||_\hsltt
        +  \sup\limits_{0\leq s\leq t}\bl  || u(s)||_\hser 
        +  ||\psi(s)||_\hsltt\brr \\
         &+ \int_0^t 
           \bbl ||\nabla \psi||_\hsltt  + ||\psi||_\hshot \bbr \d s
         +  \sup\limits_{0\leq s\leq t} ||\psi||_{s,\ler} 
         \int_0^t 
         ||\nabla u||_{s-1}^2 \d s.
     \end{aligned}
     \end{equation}
     $J_2$ and $J_3$ are bounded by
     \begin{equation}
         \begin{aligned}
             J_2\le & |||u|||\,||\nabla \psi||_\ltlt||\nabla u|| 
             + ||\lsm (u\cdot \nabla \psi)||_\ltlt ||\ls u||\\ 
             \le & ||u||_\hs 
             ( ||\nabla u||_{s-1}^2 + ||\psi||_\hshot ),\\
             J_3\leq &  C||\Delta\psi||_\ler ||\nabla u||  
             + C||\lsp \psi||_\ltlt ||\nabla u||_{s-1}\\
             \leq  & C ||\psi||_\hsltt + \frac{1}{4}||\nabla u||_{s-1}^2.
         \end{aligned}
     \end{equation}

     Finally, using integrating by parts in $R$, we obtain
     \begin{equation}
         \begin{aligned}
             J_4=&\iint \pin\gr\frac{\psi}{\pin}\gr(R_jR_k)[\D u]_{j,k} \dr \d x\\
             & +\iint \pin\gr\frac{\ls\psi}{\pin}\gr(R_jR_k)[\D \ls u]_{j,k} \dr \d x\\
             \leq & C ||\psi||_\ltho ||\nabla u|| 
             + C ||\lsm \psi||_\ltho ||\ls u|| \\
             \leq & C ||\psi||_\hshot + \frac{1}{4}||\nabla u||_{s-1}^2,
         \end{aligned}
     \end{equation}
     \begin{equation}\label{n5}
         \begin{aligned}
             J_5=&-\iint \nabla u\cdot R\psi \gr (R_jR_k)[\D u]_{j,k}\dr \d x\\
             &-\iint \lsm (\nabla u\cdot R\psi) \gr (R_jR_k)[\D \lsm u]_{j,k} \dr \d x\\
             \leq & 
             |||\psi |||_\ler
             ||\nabla u||^2              
             + ||\psi||_{s-1,\ler} 
             ||\g u||_{s-1}^2 \\
             \le & ||\psi||_\hslt ||\nabla u||_{s-1}^2.
         \end{aligned}
     \end{equation}

  By integrating (\ref{Du-define-nd}) in time and substituting the estimates (\ref{n1})-(\ref{n5}) into it, we find
  \begin{equation}
      \begin{aligned}
          \ee(t)
          \leq & C ||u_0||_\hser + C ||\psi_0||_\hsltt
        +  C \sup\limits_{0\leq s\leq t} \bl ||u(s)||_s^2
        +  ||\psi(s)||_{s,\ler}^2 \brr \\
        & + C \int_0^t  
        \bbl ||\psi||_\hsho^2 + ||\nabla \psi||_\hsltt \bbr
        \d s\\
        & + C\sup\limits_{0\leq s\leq t} 
        \bl
        ||u(s)||_s
        + ||\psi(s)||_{s,\ler}\brr
         \int_0^t \bbl 
         ||\nabla u||_{s-1}^2 
         + ||\psi||_{s,\hyi}^2
         \bbr \d s\\
         \leq & C\ey (0) + C\ey (t) + C \ey ^\frac{3}{2}(t) + C \ee^\frac{3}{2}(t),
      \end{aligned}
  \end{equation}
  and (\ref{nde2-inequ}) is proved.
 \end{proof}
 
\section{Decay estimates}\label{sec-decay}
In this section, we take advantage of the   wave structure of the velocity and explore the decay rates of (\ref{fene}).
By combining this   wave structure and the Fourier splitting method, we derive the velocity decay rate which matches that of the heat equation.
\begin{proof}[Proof of Theorem \ref{thm-decay}]
    By basic $L^2$ and $\hdyi$ energy estimates, we find
    $$
    \begin{aligned}
    &\frac{1}{2}\dt \bl ||u||_1^2 + ||\psi||_{1,\ler}^2 \brr 
    + \nu ||\g\psi||_{1,\ler}^2
        + ||\psi||_{1,\hyi}^2\\
        =& 
         - \ix \g(u\cdot \g u) : \g u\dx
            -\iint \g (u\cdot\g\psi)\cdot\g\psi\dri\dx\\
        &+\sum\limits_{k=0,1}
        \iint \g^k \divr (\-\g u\cdot R\psi) \g^k \psi\dri\dx.
    \end{aligned}
    $$
    Also, we integrate (\ref{du}) with $\D u$ to obtain
    $$
    \begin{aligned}     
        \h C ||\D u||^2 = &
        \p_t\iint ( \psi R\otimes R):\D u\dr\dx 
        -\iint ( \psi R\otimes R)\cdot\D\P(u\cdot\g u) \dr\dx \\
        &+\iint ( \psi R\otimes R)\cdot\D\P \div \tau \dr\dx 
        +\iint (u\cdot\g\psi R\otimes R):\D u\dr \dx\\
        &-\nu\iint (\Delta\psi R\otimes R):\D u\dr \dx
        -\iint (\h L \psi R\otimes R):\D u\dr\dx\\
        &+\iint(\divr (-\g u\cdot R\psi)R\otimes R):\D u\dr\dx.
    \end{aligned}
    $$

    Gathering the above equations together, we find
    \begin{equation}\label{equ-decay}
        \begin{aligned}
            &\frac{1}{2}\dt
            \bbl 
            ||u||_1^2 + ||\psi||_{1,\ler}^2
            -2 a \iint (\psi R\otimes R ):\D u \dr \dx
            \bbr\\
            & + \nu ||\g\psi||_{1,\ler}^2
            + ||\psi||_{1,\hyi}^2
            +2 a \cal C  ||\D u||^2
            = \cal L + \cal{NL},
        \end{aligned}
    \end{equation}
    where $a$ is a small positive constant that will be determined later, and
    $$\begin{aligned}
    &\begin{aligned}
        \cal L:= &
        2a\iint ( \psi R\otimes R)\cdot\D\P \div \tau \dr\dx
        -2a \nu\iint (\Delta\psi R\otimes R):\D u\dr \dx\\
        &-2a\iint (\h L \psi R\otimes R):\D u\dr\dx,
    \end{aligned}\\
   & \begin{aligned}
        \cal{NL}:= & 
        \ix \g(u\cdot \g u) : \g u\dx
        -\iint \g (u\cdot\g\psi)\cdot\g\psi\dri\dx\\
        &+\sum\limits_{k=0,1}
        \iint \divr (-\g^k u\cdot R\psi)\g^k\psi\dri\dx\\
        & -2a\iint ( \psi R\otimes R)\cdot\D\P(u\cdot\g u) \dr\dx\\
        &+2a\iint (u\cdot\g\psi R\otimes R):\D u\dr \dx\\
        & +2a\iint(\divr (-\g u\cdot R\psi)R\otimes R):\D u\dr\dx.        
    \end{aligned}
    \end{aligned}
    $$
Using integrating by parts, Lemma \ref{le-tau}, and setting $a$ small enough, we can estimate $\cal L$ by
 $$
 \begin{aligned}
 \cal L\leq & C a\iint \pin^{\frac{1}{2}}|\gr\frac{\psi}{\pin}|\,|\g u| \dr\dx\\
 & + 2a \ix |\g\psi|\,|\g\tau| \dx
 + C a\nu \iint |\Delta\psi|\,|\g u| \dri\dx\\
 \leq & \frac{\nu}{4}||\g^2 \psi||_\lerer
 + \frac{1}{4}||\psi||_{1,\hyi}^2
 +\frac{1}{2} a \cal C  ||\D u||^2.
 \end{aligned}
 $$
 Also, by the Theorem \ref{thm-exist} and the smallness of $\ep$, it is direct to deduce
 $$
 \cal {NL}\leq \frac{\nu}{4}||\g^2 \psi||_\lerer
 + \frac{1}{4}||\psi||_{1,\hyi}^2
 +\frac{1}{2} a \cal C  ||\D u||^2.
 $$

 Substituting the above estimates into (\ref{equ-decay}) yields that
 $$
 \dt f(t) + g(t) \leq 0,
 $$
 where
 $$
 \begin{aligned}
 f:= & ||u||_1^2 + ||\psi||_{1,\ler}^2
            -2 a \iint (\psi R\otimes R ):\D u \dr \dx,\\
g:= & \nu ||\g\psi||_{1,\ler}^2
            + ||\psi||_{1,\hyi}^2
            +2 a\cal C  ||\D u||^2.
 \end{aligned}
 $$
 
 To characterize the decay rate, we define
 $$
 d(t)= (\eta+ t)^s,\quad s > 0.
 $$
 By choosing $\eta$ large enough, we have
 $$
 \begin{aligned}
     & d'(t)\bbl  ||\psi||_{1,\hyi}^2 
     + ||\g u||^2
     -2a \iint (\psi R\otimes R):\D u \dr \dx\bbr \\
     \leq & d (t)\bl \frac{1}{2}  ||\psi||_{1,\hyi}^2
     +  a \cal C   ||\D u||^2\brr.
 \end{aligned}
 $$
 Hence, we obtain that 
 $$
 \dt \bl  d(t) f (t) \brr + \frac{1}{2}d(t) g(t)  \leq d'(t) ||u||^2.
 $$
 By setting 
 $S(t):=\left\{\xi: \xier\leq \frac{2}{a \cal C } \frac{d'(t)}{d(t)}\right\}$,
 we have
  \begin{equation}\label{equ-df}
       \dt \bl  d(t) f (t) \brr 
       + \frac{1}{2}a\cal C  d(t) ||\D u||^2
        \leq d'(t) \is | \hu(\xi)|^2 \d \xi .
  \end{equation}

 To estimate $\is | \hu|^2 \d \xi$, we consider (\ref{fene}) in frequency space:
 \begin{equation}\label{equ-fene-f}
  \begin{cases}
      \p_t \hu + \F(u\cdot \g u) 
      + i\xi \hat P 
      = i\xi \cdot\hat \tau ,\\
      \p_t\hp  + \F(u\cdot \g \psi) 
      +\nu \xier \hp + \h L \hp
      = \divr [-\F (\g u\cdot R(\psi+ \pin)) ],\\
      i\xi\cdot\hu = 0.
  \end{cases}   
 \end{equation}
 By integrating (\ref{equ-fene-f}) with $(\hbu,\hbp)$ and subtracting the real parts, we get
 $$
 \begin{aligned}
 &\frac{1}{2} \dt\bl |\hu|^2 + |\hp|_\lerer\brr
 +\nu \xier|\hp|_\lerer + |\hp|_\hyier\\
 = & -\re [ \F (u\cdot\g u)\cdot\hbu ]
 - \re  \intr \F(u\cdot\g \psi ) \hbp\dri 
  - \re \intr  \divr \F ( \g u\cdot R \psi )\hbp \dri.
 \end{aligned}
 $$
 By \holder inequality,
 $$
 \begin{aligned}
 &- \re  \intr \F(u\cdot\g \psi ) \hbp\dri 
  - \re \intr  \divr \F ( \g u\cdot R \psi )\hbp \dri\\
  \leq &\frac{1}{2} |\hp|_\hyier
  + C \intr |\F (u\cdot \g \psi)|^2 \dri
  + C \intr |\F (\g u\cdot R \psi)|^2 \dri,
 \end{aligned}
 $$
 which implies 
 $$
\begin{aligned}
 &\frac{1}{2} \dt\bl |\hu|^2 + |\hp|_\lerer\brr
 \leq |\re ( \F (u\cdot\g u)\cdot\hbu)|
  + C \intr |\F (u\cdot \g \psi)|^2 \dri
  + C \intr |\F (\g u\cdot R \psi)|^2 \dri.
\end{aligned}
 $$
 Integrating the above inequality in $S(t)$ and in time yields
 \begin{equation}\label{equ-f-up}
     \begin{aligned}
         &\is |\hu (t)|^2\d\xi
         + \is |\hp(t)|^2 \d \xi
         \leq \is |\hu (0)|^2\d\xi
         + \is |\hp(0)|^2 \d \xi
         + K_1+K_2+K_3,
     \end{aligned}
 \end{equation}
 where
 $$
\begin{aligned}
 K_1:= &C \it\is \intr |\F (u\cdot \g \psi)|^2 \dri \d\xi\d s,\\
 K_2:=& C \it\is\intr |\F (\g u\cdot R \psi)|^2 \dri\d\xi\d s,\\
  K_3:=&2\it \is |\re ( \F (u\cdot\g u)\cdot\hbu)|\d\xi\d s.
\end{aligned}
 $$
 
 For the initial data, we take advantage of the definition of $S(t)$ to obtain
 $$
 \begin{aligned}
 \is |\hu (0)|^2\d\xi
  + \is |\hp(0)|^2 \d \xi
    \le & \int_0^{\sqrt{\frac{2s}{a \cal C }(\eta + t)} }
    r^{N-1}\d r 
    \bl ||u(0)||_{L^1}^2
    + ||\psi(0)||_{L^1(\ler)}^2\brr\\
    \le & (\eta+ t)^{-\frac{N}{2}}
    \bl ||u(0)||_{L^1}^2
    + ||\psi(0)||_{L^1(\ler)}^2\brr.
 \end{aligned}
 $$
 For $K_1$ and $K_2$, we use Theorem \ref{thm-exist} to obtain
 $$
 \begin{aligned}
     K_1\leq & \is \d\xi \it ||u\cdot\g\psi||_{L^1}\dri\d s\\
     \le &  \is \d\xi \it 
     ||u||^2 ||\g\psi||_\lerer \d s
     \le  \varepsilon^4 (\eta + t)^{-\frac{N}{2}},\\
     K_2\leq & \is \d\xi 
     \it ||\g u\cdot R\psi||_{L^1}\dri\d s\\
     \le  & \is \d\xi \it 
     ||\g u||^2 ||\psi||_\lerer \d s
     \le \varepsilon^4 (\eta + t)^{-\frac{N}{2}}.
 \end{aligned}
 $$
 For $K_3$, we use Hardy-Littlewood-Sobolev inequality and Theorem \ref{thm-exist} to obtain
 $$
 \begin{aligned}
     K_3 &\leq 2 \it ||\Lambda^{-1}(u\cdot\g u)||\,||\g u|| \d s \\
     \leq & C \it ||u\cdot\g u||_{L^{6/5}}||\g u|| \d s
     \leq \ep C t ||\g u||^2.
 \end{aligned}
 $$
 By substituting the above estimates into (\ref{equ-f-up}), we find
 $$
 \is |\hu (t)|^2\d\xi
         + \is |\hp(t)|^2 \d \xi
         \leq C\varepsilon^4 (\eta+t)^{-\frac{N}{2}} 
         + \varepsilon Ct||\g u||^2.
 $$
  Plugging the above inequality into (\ref{equ-df}) and choosing $\ep \leq \frac{a \cal C }{2 C}$ yields
  $$
  \dt\bl d(t) f(t)\brr \leq \varepsilon^4 Cd'(t)(\eta+t)^{-\frac{N}{2}}.
  $$
  Since 
  $
  f \sim ||u||_1^2 + ||\psi||_{1,
  \ler}^2
  $,
  by integrating the above inequality in time, we obtain that
  $$
  ||u||_1^2 + ||\psi||_{1,
  \ler}^2\leq \varepsilon^2 C (\eta+t)^{-\frac{N}{2}}.
  $$

  Next, we take advantage of the higher-order decay to improve the decay rate of $||\psi||_\ler $.
  By the argument parallel to the $L^2$ case, it is not difficult to deduce that
  $$
  \begin{aligned}
  &\dt\bbl d(t) \bbl ||\g u||_1^2 + ||\g \psi||_{1,\ler}^2
            -2 a \iint \g(\psi R\otimes R ):\g\D u \dr \dx \bbr \bbr \\
            \le & d'(t) \is \xier |\hu|^2 \d \xi
            \le \ep^2 d'(t)  (\eta + t)^{-1} (\eta+t)^{-\frac{N}{2}}. 
  \end{aligned}
  $$
  Integrating the above inequality yields that
  \begin{equation}\label{est-gup}
  ||\g u||_1^2 + ||\g \psi||_{1,\ler}^2
  \le \ep^2 (\eta+ t)^{-\frac{N}{2}-1}.
  \end{equation}
 By the standard $L^2$-estimate of $\psi$, we obtain
  $$
  \begin{aligned}
  \frac{1}{2}\dt ||\psi||_\lerer
  +\nu ||\g\psi||_\lerer
  + ||\psi||_\hyier
  \leq |||\g u|||\,||\psi||_\hyier
  + C||\g u ||^2 + \frac{1}{4}||\psi||_\hyier.
  \end{aligned}
  $$
  Combining the above estimates with Theorem \ref{thm-exist} and (\ref{est-gup}) yields that
  $$
  \dt ||\psi||_\lerer + ||\psi||_\hyier 
  \le \ep^2 (\eta + t)^{-\frac{N}{2}-1}.
  $$
  Using Lemma \ref{le-poincare-r}, we have
  $$
  ||\psi||_\lerer
  \leq \ep^2 C(\eta + t)^{-\frac{N}{2}-1}.
  $$
\end{proof}
\section*{Acknowledgment}
Zheng-an Yao was partially supported by
National Natural Science Foundation of China(12126609) and
National Key Research and Development Program of China(2020YFA0712500).

\section*{Data availability}
Data sharing is not applicable to this article as no data sets were generated or analyzed during the current study.

\section*{Declarations}
\subsubsection*{Conflict of interest}
The authors declare that they have no Conflict of interest.
\bibliography{main}

\end{document}